\documentclass[12pt,xcolor=svgnames]{amsart}

\usepackage{amsmath, amsthm, amssymb, amsfonts, verbatim}

\usepackage{xspace,mathtools,bm}

\usepackage[svgnames]{xcolor}

\usepackage[pagebackref=true, colorlinks=true, citecolor=blue]{hyperref}

\allowdisplaybreaks

\hypersetup{
    colorlinks=true,
    citecolor=red,
    linkcolor=blue,
    filecolor=magenta,      
    urlcolor=cyan,
    pdfpagemode=FullScreen,
    }
\newtheorem{theorem}{Theorem}[section]

\newtheorem{lemma}[theorem]{Lemma}
\newtheorem{definition}[theorem]{Definition}

\newtheorem{remark}[theorem]{Remark}

\newcommand{\R}{\mathbb{R}}

\newcommand{\abs}[1]{\mid\!#1\!\mid}
\newcommand{\norm}[2]{\left\lVert#1\right\rVert_{#2}}

\newcommand{\Obsolete}[1]{
    }

\DeclareMathOperator\supp{supp}

\DeclareFontFamily{U}{mathx}{}
\DeclareFontShape{U}{mathx}{m}{n}{<-> mathx10}{}
\DeclareSymbolFont{mathx}{U}{mathx}{m}{n}
\DeclareMathAccent{\widehat}{0}{mathx}{"70}
\DeclareMathAccent{\widecheck}{0}{mathx}{"71}

\date{\today}
\usepackage[english]{babel}

\author{Zachary Radke}
\address{Oregon State University, Department of Mathematics}
\email{radkeza@oregonstate.edu}
\title[Logarithmically Regularized Euler Equations]{Sharp Global Well-Posedness for the Logarithmically Regularized $2$D Euler Equations}
\date{\today}

\begin{document}

\maketitle
\begin{abstract}
We consider solutions to the two-dimensional Euler equations with a logarithmically regularized velocity given by $u=\nabla^{\perp}(-\Delta)^{-1}\ln^{-\gamma}(|\nabla|+e)\omega$ and initial vorticity in $\dot{H}^1 \cap \dot{H}^{-1}(\R^2)$. We show that if $\gamma>1/2$, then global well-posedness holds.
\end{abstract}

\section{Introduction}
The incompressible Euler equations on $\mathbb{R}^2$ in velocity form are given by
\begin{equation}\label{VelocityEquation}\tag{E}
\begin{cases}
    \partial_tu+(u\cdot\nabla) u = -\nabla p \quad &\text{on }\mathbb{R}^2\times(0,\infty),
    \\ \nabla\cdot u=0 \quad &\text{on }\mathbb{R}^2\times[0,\infty),
    \\ u|_{t=0}=u_0 \quad &\text{on }\mathbb{R}^2,    
\end{cases}    
\end{equation}
where $u:\mathbb{R}^2\times [0,\infty)\rightarrow \mathbb{R}^2$ denotes the velocity and $p:\mathbb{R}^2\times[0,\infty)\to\mathbb{R}$ denotes the internal scalar pressure. 
The 2D incompressible Euler equations can be recast in terms of the \textit{vorticity}, defined by $$\omega:=-\nabla^\perp \cdot u=\partial_{x_1}u_2-\partial_{x_2}u_1.$$ One can show using the incompressibility condition that $(\ref{VelocityEquation})$ is equivalent to the 2D vorticity equation, given by
\begin{equation}\label{VorticityEquation}\tag{V} 
\begin{cases}
    \partial_t\omega+(u\cdot\nabla)\omega = 0 \quad &\text{on }\mathbb{R}^2\times(0,\infty),
    \\ u = \nabla^\perp(-\Delta)^{-1}\omega  &\text{on }\mathbb{R}^2\times[0,\infty),
    \\ \omega|_{t=0}=\omega_0 &\text{on }\mathbb{R}^2.
\end{cases}
\end{equation}
It is well known that the system \eqref{VelocityEquation} is globally well posed in $W^{s,p}(\R^2)$ for $s>\frac{2}{p}+1$ and $1<p<\infty$, see \cite{katoponce86} for example. It was a long-standing open question as to what happened at the scaling-critical set of parameters $s=\frac{2}{p}+1$. In the work \cite{vishik98},  Vishik considers the Besov space $B^{2/p+1}_{p,1}(\R^2)$ (see Definition \ref{besov space}) for $1<p<\infty$, and establishes global well-posedness for \eqref{VelocityEquation}. Similarly, the work of Chae in \cite{chae03} established global well-posedness for \eqref{VelocityEquation} for initial velocity in $F^{3}_{1,q}(\R^2)$ for any $1<q<\infty.$ The key idea in both of these works is recognizing that both $B^{2/p+1}_{p,1}(\R^2)$ and $F^{3}_{1,q}(\R^2)$ embed into $B^1_{\infty,1}(\R^2)$, allowing for a priori control of the velocity gradient. Then, using paradifferential calculus, Vishik establishes a composition estimate of a $B^0_{\infty,1}(\R^2)$ function with a measure-preserving transformation, which, using the transport structure of \eqref{VorticityEquation}, shows the $B^0_{\infty,1}(\R^2)$ norm of the vorticity grows at most exponentially in time, establishing global well-posedness.
\\
\\
Instead of considering initial velocity in function spaces that embed into $B^1_{\infty,1}(\R^2)$ (allowing for global well-posedness), another approach to salvaging well-posedness is to regularize the Biot-Savart law. In particular, we consider the logarithmically regularized Euler equations given by 
\begin{equation}\label{logVorticityEquation}\tag{V$_{\gamma}$} 
\begin{cases}
    \partial_t\omega+(u\cdot\nabla)\omega = 0 \quad &\text{on }\mathbb{R}^2\times(0,\infty),
    \\ u = \nabla^\perp(-\Delta)^{-1}T_{\gamma}\omega  &\text{on }\mathbb{R}^2\times[0,\infty),
    \\ \omega|_{t=0}=\omega_0 &\text{on }\mathbb{R}^2,
\end{cases}
\end{equation}
where $T_{\gamma}=T_{\gamma}(|\nabla|)$ is an operator defined by the Fourier\footnote{Throughout this paper we use the convention 
$
    \hat{f}(\xi) := \int_{\mathbb{R}^d} f(x) e^{-2\pi i x \cdot \xi}\, dx $ to denote the Fourier transform of a sufficiently nice function. $f:\mathbb{R}^d\to\mathbb{R}$.}multiplier 
\begin{equation}
\label{Fourier Multiplier}
\mathcal{F}\left(T_{\gamma}\omega\right)(\xi)
    =\frac{1}{\ln^{\gamma}(e+|\xi|)}\mathcal{F}(\omega)(\xi).
\end{equation}
This model provides logarithmic smoothing of the velocity field, and one has a priori control of the velocity gradient for $\gamma>\frac{1}{2}$ and $\omega \in H^1(\R^2)$. Indeed, note that 
\begin{equation}
\label{apriori velocity gradient}
|\nabla u(x)| \leq \int_{\R^2} \ln^{-\gamma}(e+|\xi|)|\hat{\omega}\,(\xi)|\,d\xi \leq C\norm{\omega}{L^2}+ C\left(\int_1^\infty \frac{dt}{t\ln^{2\gamma}(e+t)} \right)^{1/2}\norm{\omega}{\dot{H}^1},
\end{equation}
and the latter integral is finite when $\gamma >\frac{1}{2}$. Utilizing this, the authors in \cite{CW12} establish local well-posedness of \eqref{logVorticityEquation} for $\omega_0 \in H^1(\R^2)$ and $\gamma >\frac{1}{2}$. However, the authors were not able to establish global well-posedness in this class. In particular, they prove global existence for $\omega_0 \in W^{1,p}(\R^2)$ for $p>2$ by exploiting the following bound from \cite{KT2000}
$$
\norm{f}{L^{\infty}} \leq C\left(1+\norm{f}{\text{BMO}}\ln(e+\norm{f}{W^{1,p}}\right)), \quad \text{ for $p>2$}.$$
Later, in the work \cite{DL15}, Dong and Li establish 
$$
\norm{\nabla\nabla^{\perp}(-\Delta)^{-1}\ln^{-\gamma}(|\nabla|+e)f}{L^{\infty}} \leq C_1\ln(\norm{f}{H^1}+e)\sup_{2 \leq p<\infty} \frac{\norm{f}{L^p}}{\sqrt{p}} , \quad \text{for } \gamma\geq \frac{3}{2},
$$
where $C_1$ is some absolute constant. This provides the necessary logarithmic control of the velocity gradient, establishing global well-posedness for $\gamma \geq \frac{3}{2}$. Ill-posedness for \eqref{logVorticityEquation} has been studied as well. In the work \cite{kwon21}, for $0<\gamma\leq \frac{1}{2}$, the author establishes strong ill-posedness. Without a priori control of the velocity gradient, Kwon shows there is smooth initial data which is arbitrarily small in $H^1(\R^2)$, yet can be perturbed by a highly oscillatory function creating norm inflation in short time. Kwon then patches together many local solutions, creating larger norm inflation, establishing instantaneous blow-up of the $H^1(\R^2)$-norm.
\\
\\
To the best of the author's knowledge, the case when $\frac{1}{2}<\gamma<\frac{3}{2}$ remains open. In this work, we show that global well-posedness holds throughout this range. To do this, using the logarithmic smoothing of the velocity, we can bound the $L^{\infty}(\R^2)$ norm of the velocity gradient by $\norm{\omega_0}{L^2} + \norm{((1+j)^{-\gamma})\|\Delta_j\omega\|_{L^{\infty}})_j}{l^1(\mathbb{Z}_{\geq 0})}$, which is finite by Cauchy-Schwarz provided that $\gamma >\frac{1}{2}$. In particular, we note that 
\begin{equation}
\norm{\nabla u}{L^{\infty}} \leq C\norm{\omega_0}{L^2} + C_{\gamma}\norm{\omega(t)}{B^0_{\infty,2}}, \quad \text{for $\gamma>\frac{1}{2}$},
\end{equation}
and then show that the $B^0_{\infty,2}(\R^2)$-norm of the vorticity can grow at most exponentially in time. We state our main theorem below.
\begin{theorem}
\label{main theorem}
Let $\gamma >\frac{1}{2}$ and $\omega_0 \in \dot{H}^1 \cap \dot{H}^{-1}(\R^2)$. Then there exists a unique solution to \eqref{logVorticityEquation} satisfying
$$
\omega \in C^0([0,\infty); \dot{H}^{-1}\cap \dot{H}^1(\R^2)).
$$
\end{theorem}

\begin{remark}  We remark, as the authors did in \cite{DL15}, that the role of $\omega_0 \in \dot{H}^{-1}(\R^2)$ is not artificial. This is since the local well-posedness argument in \cite{DL15} needs $u \in L^{\infty}(\R^2)$ and $\omega_0 \in H^1(\R^2)$ is not sufficient to handle the low frequencies arising from the Biot-Savart law (however, note that $\norm{\omega_0}{L^2}^2 \leq \norm{\omega_0}{\dot{H}^{-1}}\norm{\omega_0}{\dot{H}^1}$). 
\end{remark}

\subsection{Notation and Organization of Paper}
We first introduce the notation used throughout this paper. We let $C>0$ denote a positive constant that may change from line to line. If we wish to highlight the dependence of the constant $C$ on a parameter $b$, we write $C=C_b$. If $X$ and $Y$ are normed vector spaces, we write $X \hookrightarrow Y$ if $X \subseteq Y$ and $\norm{f}{Y}\leq C\norm{f}{X}$ for all $f \in X$. We write $\ln$ for the natural logarithm.
\\
\\
We organize the paper as follows. Section \ref{section prelim} covers preliminaries such as Littlewood-Paley theory, function spaces, and estimates on the velocity gradient. Section \ref{section proof} covers the proof of the main theorem.
\section{Preliminaries}
\label{section prelim}

\subsection{Littlewood-Paley Theory} 
Before defining our function spaces, we begin with a brief overview of the Littlewood-Paley operators. We closely follow \cite{Chemin95}.\\
\\
It is well known that if $\mathcal{C}$ is the annulus of center $0$, inner radius $3/5$, and outer radius $5/3$, then there exist two positive radial functions $\hat{\chi} \in C_c^{\infty}(B_{5/6}(0))$ and $\hat{\varphi}\in C_c^{\infty}(\mathcal{C})$ such that, for $j \in \mathbb{Z}$, we set $\varphi_j(x)=2^{jd}\varphi(2^jx)$ and
$$
\hat{\chi}+\sum_{j\geq 0} \hat{\varphi}_j = \hat{\chi} + \sum_{j \geq 0} \hat{\varphi}(2^{-j}\cdot)
 \equiv 1.$$
For $n\geq 0$ we define $\chi_n$ by $\chi_n(x)=2^{nd}\chi(2^nx)$, so that its Fourier transform satisfies
$$
\hat{\chi}_n(\xi) = \hat{\chi}(2^{-n}\xi) \text{ for } \xi \in \R^d.
$$
For a tempered distribution $f \in \mathcal{S}'(\R^d)$, we define the operator $S_n$ as 
$$
S_n f =\chi_n \ast f.
$$
Formally, $S_n$ can be thought of as a projection in Fourier space into the ball of radius of order $2^n$. For $f \in \mathcal{S}'(\R^d)$ we define the inhomogeneous Littlewood-Paley operators $\Delta_j$ by 
\begin{equation*}
\label{LP operator}
\Delta_j f
=
\begin{cases}
    0, \hspace{1cm} &\text{if }j < -1,\\
    \chi \ast f, &\text{if } j=-1, \\
    \varphi_j \ast f, &\text{if } j\geq 0,
\end{cases}
\end{equation*}
and for all $j\in \mathbb{Z}$ we define the homogeneous Littlewood-Paley operators $\dot{\Delta}_j$ by
$$
\dot{\Delta}_j f = \varphi_j \ast f.
$$
 Up to constants, one can formally think of $\dot{\Delta}_j$ as a projection in Fourier space into the annulus of inner and outer radius of order $2^j$. For any $u \in \mathcal{S}'(\R^d)$, we have that $u = \lim_{n \to \infty} S_nu$ where the limit holds in $\mathcal{S}'(\R^d)$. Or equivalently
$$
u = \sum_{j\geq -1} \Delta_j u \text{ in } \mathcal{S}'(\R^d).
$$
One benefit of using Littlewood-Paley operators can be seen in Bernstein's Lemma. Heuristically, Bernstein's Lemma states that if the support of $\hat{u}$ is included in some ball $B_{\lambda}(0)$, then a derivative of $u$ costs no more than a power of $\lambda$ when taking an $L^p$ norm. This will be made precise below. We let $\mathcal{C}_{a,b}(0)$ denote the annulus with inner radius $a$ and outer radius $b$.

\begin{lemma}(Bernstein's Lemma) Let $(r_1,r_2)$ be a pair of strictly positive real numbers such that $r_1<r_2$ and let $\lambda$ denote a positive real number. Then there exists a constant $C>0$, such that, for every integer $k$ with $k \geq 0$, every pair of real numbers $(a,b)$ satisfying $b \geq a \geq 1$, and every function $u \in L^a(\R^d)$,
\begin{enumerate}

\item 
\label{Bernstein in Ball}
If $\supp{\hat{u}} \subset B_{r_1\lambda}(0)$ then
\begin{equation*}
 \sup_{\abs{\alpha} = k} \norm{\partial^{\alpha}u}{L^b} \leq C^k\lambda^{k+d(\frac{1}{a}-\frac{1}{b})}\norm{u}{L^a}.
\end{equation*}\\
\item 
\label{Bernstein in Annulus}
Furthermore, if $\supp{\hat{u}} \subset \mathcal{C}_{r_1\lambda, r_2\lambda}(0)$ then
\begin{equation*}
 C^{-k}\lambda^k\norm{u}{L^a} \leq  \sup_{\abs{\alpha} = k} \norm{\partial^{\alpha}u}{L^a} \leq C^k\lambda^k\norm{u}{L^a}.
\end{equation*}
\end{enumerate}
\end{lemma}
\noindent
When estimating the gradient of the velocity, we make use of the following lemma from \cite{DL15}.
\begin{lemma}
\label{generalized bernstein}
Let $m \in C^{d+1}(\R^d \setminus \{0\})$ be such that for any dyadic $N>0$, there exists a constant $\tilde{m}(N)$ so that 
$$
\sup_{\frac{N}{8} \leq |\xi| \leq 8N} |\partial_{\xi}^{\alpha}m(\xi)| \leq C_d \frac{\tilde{m}(N)}{N^{|\alpha|}}, \quad \text{ for all } |\alpha| \leq d+1.
$$
Let $T_m$ be the associated operator defined by the Fourier multiplier
$$
\mathcal{F}(T_mf)(\xi) = m(\xi)\mathcal{F}(f)(\xi).
$$
Then for any integer $j \geq 0$ and $1 \leq p \leq \infty$ we have 
$$
\norm{T_m \Delta_jf}{L^p} \lesssim_{d} \tilde{m}(2^j)\norm{\Delta_jf}{L^p}.
$$
\end{lemma}
\noindent
Next we define Besov spaces, which provide us with the necessary tools to estimate the velocity gradient.

\begin{definition}
\label{besov space}
Let $p,q \in [1,\infty]$ and $ s \in \R$. We define the Besov space $B^{s}_{p,q}(\R^d)$ as the set of $f \in \mathcal{S}'(\R^d)$ such that the following norm (with the obvious modification when $q=\infty$) is finite
\begin{equation}
\label{inhomogeneous Besov space}
\norm{f}{B^{s}_{p,q}} = \left(\sum_{j \geq -1} 2^{jsq}\norm{\Delta_jf}{L^p}^q\right)^{1/q}.\\
\end{equation}
\end{definition}
\noindent
As a trivial consequence of Bernstein's Lemma, $H^1(\R^2) \hookrightarrow B^0_{\infty,2}(\R^2)$. Indeed, note that
$$
\norm{f}{B^0_{\infty,2}}^2 = \sum_{j \geq -1} \norm{\Delta_jf}{L^{\infty}}^2  \leq C\sum_{j \geq -1} 2^{2j}\norm{\Delta_jf}{L^2}^2 = C\norm{f}{H^1}^2,
$$
where we used equivalence of $\norm{\cdot}{H^s}$ and $\norm{\cdot}{B^{s}_{2,2}}$ (see \cite{Chemin95} for example) in the last step.
\subsection{Estimates for the Velocity Gradient}
\begin{lemma}
\label{velocity gradient 1}
Let $\gamma>\frac{1}{2}$ and $f \in L^2 \cap B^0_{\infty,2}(\R^2)$. 
Then
$$
\norm{\nabla \nabla^{\perp}(-\Delta)^{-1}T_{\gamma}f}{L^{\infty}} \leq C\norm{f}{L^2} + C_{\gamma}\norm{f}{B^0_{\infty,2}}.
$$
\end{lemma}
\begin{proof}
Using a Littlewood-Paley decomposition, we write 
\begin{equation}
\label{vg1 e1}
\norm{\nabla \nabla^{\perp}(-\Delta)^{-1}T_{\gamma}f}{L^{\infty}} \leq \norm{\Delta_{-1}\nabla \nabla^{\perp}(-\Delta)^{-1}T_{\gamma}f}{L^{\infty}} + \sum_{j \geq 0}\norm{\Delta_j\nabla \nabla^{\perp}(-\Delta)^{-1}T_{\gamma}f}{L^{\infty}}. 
\end{equation}
To estimate the low frequencies, we see that 
\begin{equation}
\label{vg1 e2}
\norm{\Delta_{-1}\nabla \nabla^{\perp}(-\Delta)^{-1}T_{\gamma}f}{L^{\infty}} \leq \int_{\R^2} \left|\hat{\chi}(\xi) \ln^{-\gamma}(e+|\xi|) \hat{\omega}(\xi)\right|\,d\xi \leq C\norm{\hat{\chi}}{L^2}\norm{\hat{f}}{L^2} \leq C\norm{f}{L^2},
\end{equation}
where in the second inequality we used that $\ln^{-\gamma}(e+|\xi|)$ is bounded near the origin. To estimate the high frequencies, we apply Lemma \ref{generalized bernstein} to obtain
\begin{equation}
\label{vg1 e3}
\sum_{j \geq 0} \norm{\Delta_j \nabla \nabla^{\perp}(-\Delta)^{-1}T_{\gamma}f}{L^{\infty}} \leq C\sum_{j \geq 0}(1+j)^{-\gamma}\norm{\Delta_jf}{L^\infty} \leq C\left(\sum_{j \geq 0} (1+j)^{-2\gamma}\right)^{\frac{1}{2}} \norm{f}{B^0_{\infty,2}}.
\end{equation}
The last sum is finite when $\gamma >\frac{1}{2}$. Combining \eqref{vg1 e1}, \eqref{vg1 e2}, and \eqref{vg1 e3} yields the result.
\end{proof}
We now state the key theorem we use to control the $B^0_{\infty,2}(\R^2)$ norm of the vorticity. The following result comes from \cite{CF23}. The authors state that the proof is classical and do not include it. However, since this estimate is central to our main theorem, we supply a proof. 

\begin{theorem}\label{Key Besov estimate} Let $r \in [1,\infty]$, $a_0 \in L^2 \cap B^0_{\infty,r}(\R^d)$, and $a$ be a solution to
\begin{equation}
\label{transport}
\tag{$\text{Tr}$}
\begin{cases}
\partial_ta + (v\cdot \nabla)a= 0,\\
\text{div }v=0,\\
a(0,x)=a_0(x) ,
\end{cases}
\end{equation}
such that 
$
a \in C^0([0,T]; L^2 \cap B^0_{\infty,r}(\R^d))$ and $ 
 v \in L^1([0,T]; W^{1,\infty}(\R^d)).$
Then for any $t \leq T$ one has
\begin{equation}
\label{key besov estimate eqn}
\norm{a(t)}{B^0_{\infty,r}} \leq C\norm{a_0}{B^0_{\infty,r}} \left(1+\int_0^t \norm{\nabla v(\tau)}{L^{\infty}}\,d\tau\right).
\end{equation}
\end{theorem} 

\begin{proof}
For a fixed solution $a\in C^0([0,T]; L^2 \cap B^0_{\infty,r}(\R^d))$, we consider the system
\begin{equation}
\label{Transport q}
\tag{$\text{Tr}_q$}
\begin{cases}
\partial_ta_q + (v\cdot \nabla)a_q= 0,\\
\text{div }v=0,\\
a_q(0,x)=\Delta_qa_0(x).
\end{cases}
\end{equation}
Using the method of characteristics, we see that 
$
a_q(t,x)=\Delta_qa_0(\phi^{-1}_t(x)),
$
where $\phi$ satisfies the ODE
$$
\frac{d}{dt} \phi_t(x) = v(t,\phi_t(x)), \quad \phi(0,x)=x.
$$
Furthermore, since $\text{div }v=0$, we see that $\phi_t$ is measure preserving and 
\begin{equation}
    \label{TRq L2 convergence}
    \begin{split}
    \norm{a(t)-\sum_{q=-1}^n a_q(t)}{L^2} &= \norm{\left(a_0-  \sum_{q=-1}^n \Delta_qa_0\right)(\phi^{-1}_t(x))}{L^2}\\
    &= \norm{a_0 - \sum_{q=-1}^n\Delta_qa_0}{L^2} \to 0 \quad \text{as }n \to \infty.
    \end{split}
\end{equation}
Furthermore, since $\text{div }v=0$, solutions of \eqref{Transport q} satisfy a maximum principle and
\begin{equation}
\label{TRq eq1}
\norm{\Delta_j a_q(t)}{L^\infty} \leq C\norm{a_q(t)}{L^{\infty}} \leq C\norm{\Delta_qa_0}{L^{\infty}}.
\end{equation}
Next, we show 
\begin{equation}
\label{TRq eq 2}
\norm{\Delta_ja_q(t)}{L^\infty} \leq Ce^{CV(t)}2^{-|j-q|}\norm{\Delta_qa_0}{L^{\infty}},
\end{equation}
where 
$$
V(t) = \int_0^t \norm{\nabla v(\tau)}{L^{\infty}}\,d\tau.
$$
Since $v \in L^1([0,T];W^{1,\infty}(\R^d))$, from Gr\"{o}nwall's Lemma, we see that 
\begin{equation}
\label{flow bounds}
\norm{D\phi_t}{L^{\infty}} + \norm{D\phi_{t}^{-1}}{L^{\infty}} \leq Ce^{CV(t)},
\end{equation}
where $Df(x)=\left(\frac{\partial f^i}{\partial_{x_k}}(x)\right)_{i,k=1}^d$ is the Jacobian matrix. To establish \eqref{TRq eq 2} we consider three cases, when $j>q, j<q$, and $j=q$. \\
\\
\textbf{Case 1 ($j>q$):} Since $q \geq -1$, we see that $j \geq 0$ and 
$$
\int_{\R^d} \varphi_j(x)\,dx = \int_{\R^d} \varphi_j(x)e^{2\pi i x \cdot 0}\,dx = \hat{\varphi_j}(0)=0.
$$
From the mean value theorem 
\begin{equation}
\label{case 1 eq 1}
\begin{split}
|\Delta_ja_q(t,x)| &= \left| \int_{\R^d} \varphi_j(y) (a_q(t,x-y)-a_q(t,x))\,dy\right|\\
& \leq \norm{\nabla a_q(t)}{L^{\infty}} \int_{\R^d} |\varphi_j(y)y|\,dy\\
&\leq C2^{-j}\norm{\nabla a_q(t)}{L^{\infty}}.
\end{split}
\end{equation}
From the transport structure of \eqref{Transport q} and the chain rule, we see that 
$$
\nabla a_q(t,x)=\nabla \left((\Delta_qa_0)(\phi_t^{-1}(x))\right) =(D\phi_{t}^{-1}(x))^T(\nabla\Delta_qa_0)(\phi_t^{-1}(x)).
$$
Using \eqref{flow bounds} we see that 
\begin{equation}
\label{case 1 eq 2}
\norm{\nabla a_q(t)}{L^{\infty}} \leq \norm{D\phi_t^{-1}}{L^{\infty}}\norm{\nabla \Delta_qa_0(\phi_t^{-1})}{L^{\infty}}\leq Ce^{CV(t)}2^q\norm{\Delta_qa_0}{L^{\infty}}.
\end{equation}
Combining \eqref{case 1 eq 1} and \eqref{case 1 eq 2} yields
\begin{equation}
\label{case 1}
\norm{\Delta_j a_q(t)}{L^{\infty}} \leq Ce^{CV(t)}2^{q-j}\norm{\Delta_qa_0}{L^{\infty}}.
\end{equation}
\textbf{Case 2 ($q>j$):} Note that 
\begin{equation}
\label{case 2 eq 0}
\Delta_qa_0(x)=\text{div}(-\nabla(-\Delta)^{-1}\Delta_qa_0)(x).
\end{equation}
We first consider when $j \geq 0$. Then, since the flow is measure-preserving, by a change of variables and \eqref{case 2 eq 0}
\begin{equation}
\label{case 2 eq 1}
\begin{split}
|\Delta_ja_q(t,x)|&= \left|\int_{\R^d} \varphi_j(x-y)(\Delta_qa_0)(\phi_t^{-1}(y))\,dy\right|\\
&=\left|\int_{\R^d}\varphi_j(x-\phi_t(z))\text{div}(-\nabla(-\Delta)^{-1}\Delta_qa_0)(z)\,dz\right|\\
&= \left|\int_{\R^d}\nabla_z(\varphi_j(x-\phi_t(z)))\cdot(\nabla(-\Delta)^{-1}\Delta_qa_0)(z)\,dz\right|\\
&\leq C2^{-q}\norm{\Delta_qa_0}{L^{\infty}} \int_{\R^d} |\nabla_z(\varphi_j(x-\phi_t(z)))|\,dz\\
&\leq C2^{-q}\norm{\Delta_qa_0}{L^{\infty}} \norm{D\phi_t}{L^{\infty}} \int_{\R^d} |(\nabla\varphi_j)(x-\phi_t(z))|\,dz\\
&\leq Ce^{CV(t)}2^{j-q}\norm{\Delta_qa_0}{L^{\infty}}.
\end{split}
\end{equation}
A similar argument for $j=-1$ gives
\begin{equation}
\label{case 2 eq 2}
\begin{split}
\norm{\Delta_{-1}a_q(t)}{L^{\infty}} & \leq 2^{-q}\norm{\Delta_qa_0}{L^{\infty}}\norm{D\phi_t}{L^{\infty}}\norm{\nabla \chi}{L^1}\\
&\leq Ce^{CV(t)}2^{-q}\norm{\Delta_qa_0}{L^{\infty}}\\
&\leq Ce^{CV(t)}2^{-|1-q|}\norm{\Delta_qa_0}{L^{\infty}}.
\end{split}
\end{equation}
\textbf{Case 3 ($j=q$):} For the final case, we simply use that 
\begin{equation}
\label{case 3}
\norm{\Delta_ja_q(t)}{L^{\infty}} \leq C\norm{a_q(t)}{L^{\infty}} \leq C\norm{\Delta_qa_0}{L^{\infty}}.
\end{equation}
Combining \eqref{case 1},\eqref{case 2 eq 1}, \eqref{case 2 eq 2}, and \eqref{case 3} yield \eqref{TRq eq 2}. 
\\
\\
Now, from \eqref{TRq eq1} and \eqref{TRq eq 2} we see that 
\begin{equation}
\label{TRq eq3}
\norm{\Delta_ja_q(t)}{L^{\infty}} \leq C\min\left\{1,e^{CV(t)}2^{-|j-q|}\right\}\norm{\Delta_qa_0}{L^{\infty}}.
\end{equation}
Since $a(t)=\sum_{q \geq -1} a_q(t)$ in $L^2(\R^d)$, from Bernstein's Lemma one has $\Delta_ja(t)=\Delta_j\sum_{q \geq -1} a_q(t)$ in $L^{\infty}(\R^d)$. It follows that 
\begin{equation}
\label{TRq eq 4}
\norm{\Delta_ja(t)}{L^{\infty}} \leq C\sum_{q \geq -1} \min\left\{1,e^{CV(t)}2^{-|j-q|}\right\}\norm{\Delta_qa_0}{L^{\infty}}.
\end{equation}
Computing the $B^0_{\infty,r}(\R^d)$-norm (with the obvious modification when $r=\infty$), from \eqref{TRq eq 4} we find that 
\begin{equation}
\label{TR eq1}
\begin{split}
\norm{a(t)}{B^0_{\infty,r}} &= \left(\sum_{j \geq -1} \norm{\Delta_ja(t)}{L^{\infty}}^r\right)^{\frac{1}{r}}\\
&\leq C\left(\sum_{j \geq -1}\left(\sum_{q \geq -1} \min\left\{1,e^{CV(t)}2^{-|j-q|}\right\}\norm{\Delta_qa_0}{L^{\infty}}\right)^r\right)^{\frac{1}{r}}\\
&\leq C\norm{a_0}{B^0_{\infty,r}} \sum_{q \in \mathbb{Z}} \min\left\{1,e^{CV(t)}2^{-|q|}\right\},
\end{split}
\end{equation}
where in the last inequality we used Young's convolution inequality for series. We now split the sum when $e^{CV(t)}2^{-|q|}= 1$ and estimate the two terms separately. Hence, we let $N=\left\lceil \frac{CV(t)}{\ln(2)}\right\rceil$ and observe that
\begin{equation}
\label{Tr eq2}
\begin{split}
\sum_{q \in \mathbb{Z}} \min\left\{1,e^{CV(t)}2^{-|q|}\right\}&= \sum_{|q| \leq N}\min\left\{1,e^{CV(t)}2^{-|q|}\right\} +   \sum_{|q| > N}\min\left\{1,e^{CV(t)}2^{-|q|}\right\}\\
&\leq \sum_{|q| \leq N} 1 +   e^{CV(t)}\sum_{|q| > N} 2^{-|q|}\\
&= 2N +1 +Ce^{CV(t)}2^{-N}\\
&\leq C(1+V(t)),
\end{split}
\end{equation}
which concludes the proof.
\end{proof}

\section{Proof of Main Theorem}
\label{section proof}
\noindent
Local well-posedness has already been established in \cite{CW12} and \cite{DL15}. Hence, we only need to show we can extend the solution globally in time. It suffices to control the $\dot{H}^{-1}\cap \dot{H}^{1}(\R^2)$-norm of the solution. We let $T^{*} \in (0,\infty]$ denote the maximal time of existence for the local solution $\omega \in C^0([0,T^*);\dot{H}^{-1}\cap \dot{H}^{1}(\R^2))$. Following the footsteps of \cite{DL15}, we see for all $0 \leq t< T^*$
\begin{equation}
\label{local estimate}
\begin{split}
\frac{d}{dt} \norm{\omega(t)}{\dot{H}^{-1}}^2 &\leq C_{\gamma}\left( \norm{\omega_0}{L^2}\norm{\omega(t)}{\dot{H}^{-1}}^2 + \norm{\omega_0}{L^2}^2\norm{\omega(t)}{\dot{H}^{-1}}\right),\\
\frac{d}{dt} \norm{\omega(t)}{\dot{H}^{1}}^2 &\leq C \norm{\nabla\nabla^{\perp} (-\Delta)^{-1}T_{\gamma}\omega(t)}{L^{\infty}} \norm{\omega(t)}{\dot{H}^1}^2.
\end{split}
\end{equation}
From Gr\"{o}nwall's Lemma, we see that the $\dot{H}^{-1}(\R^2)$ norm is controlled for all time. Hence we only need to control the $\dot{H}^1(\R^2)$ norm of the solution.
Using Lemma \ref{velocity gradient 1} and the conservation of $\norm{\omega(t)}{L^2}$, we bound the velocity gradient by
\begin{equation}
\label{MTP e1}
\norm{\nabla\nabla^{\perp} (-\Delta)^{-1}T_{\gamma}\omega(t)}{L^{\infty}}  \leq  C\norm{\omega_0}{L^2} + C_{\gamma}\norm{\omega(t)}{B^0_{\infty,2}}.
\end{equation}
 Theorem \ref{Key Besov estimate} gives
\begin{equation}
\label{MTP e2}
\norm{\omega(t)}{B^0_{\infty,2}} \leq C\norm{\omega_0}{B^0_{\infty,2}}\left(1+\int_0^t \norm{\nabla \nabla^{\perp}(-\Delta)^{-1}T_{\gamma}\omega(\tau)}{L^{\infty}}\,d\tau\right).
\end{equation}
Hence, combining \eqref{MTP e1} and \eqref{MTP e2} and applying Gr\"{o}nwall's Lemma yields
\begin{equation}
\label{MTP e3}
\norm{\nabla\nabla^{\perp}(-\Delta)^{-1}T_{\gamma}\omega(t)}{L^\infty} \leq (C\norm{\omega_0}{L^2}+C_{\gamma}\norm{\omega_0}{B^0_{\infty,2}})e^{C_{\gamma}\norm{\omega_0}{B^0_{\infty,2}}t}.
\end{equation}
Applying Gr\"{o}nwall's Lemma to \eqref{local estimate}$_2$ and using \eqref{MTP e3} yields
\begin{equation}
\label{MTP e4}
\begin{split}
\norm{\omega(t)}{\dot{H^1}} &\leq \norm{\omega_0}{\dot{H}^1}\exp\left(C\int_0^t \norm{\nabla \nabla^{\perp}(-\Delta)^{-1}T_{\gamma}\omega(\tau)}{L^{\infty}}\,d\tau\right) \\&\leq \norm{\omega_0}{\dot{H}^1}e^{(C\norm{\omega_0}{L^2}t+C_{\gamma}\norm{\omega_0}{B^0_{\infty,2}}t)e^{C_{\gamma}\norm{\omega_0}{B^0_{\infty,2}}t}}.
\end{split}
\end{equation}
Hence, if $T^*$ is finite, we have shown
$$
\sup_{0 \leq t < T^*} \left(\norm{\omega(t)}{\dot{H}^1} + \norm{\omega(t)}{\dot{H}^{-1}} \right)<\infty.
$$
Hence, by a standard bootstrapping argument, we can extend the solution in time past $T^*$, contradicting the maximality of $T^*$. It follows that $T^*=\infty.$

\bibliographystyle{alpha}
\bibliography{logglobal}

\end{document}